\documentclass[pdflatex,sn-mathphys-num]{sn-jnl}

\newcommand{\can}{\operatorname{can}}
\newcommand{\wq}{\mathbf{w}}
\newcommand{\uq}{\mathbf{u}}
\newcommand{\vq}{\mathbf{v}}
\newcommand{\aq}{\mathbf{a}}
\newcommand{\sq}{\mathbf{s}}
\newcommand{\eq}{\mathbf{e}}
\newcommand{\fq}{\mathbf{f}}

\usepackage[T1]{fontenc}
\usepackage[utf8]{inputenc}
\usepackage{bbm}
\usepackage{amsmath, amsthm, amssymb, amsfonts}
\usepackage{hyperref}
\usepackage{enumitem}
\usepackage{mathrsfs}
\usepackage{tikz}
\usepackage{graphicx}

\newtheorem{theorem}{Theorem}[section]
\newtheorem{lemma}[theorem]{Lemma}
\newtheorem{proposition}[theorem]{Proposition}
\newtheorem{corollary}[theorem]{Corollary}
\theoremstyle{definition}
\newtheorem{definition}[theorem]{Definition}

\theoremstyle{remark}
\newtheorem{remark}[theorem]{Remark}

\setlist[itemize]{
  topsep=6pt,
  partopsep=0pt,
  itemsep=2pt,
  parsep=0pt,
  labelindent=\parindent,
  leftmargin=*,
  align=left
}

\setlist[enumerate,1]{
  topsep=6pt,
  partopsep=0pt,
  itemsep=2pt,
  parsep=0pt,
  labelindent=\parindent,
  leftmargin=*,
  align=right
}

\begin{document}

\title[Zero Cancellation and Equation Structure in Kiselman's Semigroup]{Zero Cancellation and Equation Structure in Kiselman's Semigroup}

\author{\fnm{Luka} \sur{Andrenšek}}

\abstract{
   \normalfont\unboldmath
   We investigate equations in Kiselman's semigroup $K_n$, generated by $a_1, \dots, a_n$. 
   Let $f$ denote the zero element of $K_n$.
   We prove that if $y \in K_n$ lies in the subsemigroup generated by $a_2, \dots, a_n$, then $x y = f$ implies $x = f$.
   In contrast, the equation $x a_1 = f$ admits non-trivial solutions.
   We describe the solution set of this equation, show that its cardinality is $1 + |K_{n-1}|$, and study its algebraic structure.
   Moreover, we show that $|K_{2n+1}|$ is even, whereas $|K_{2n}|$ is odd.
}


\maketitle

\section{Introduction}

Let $n$ be a positive integer.
The \emph{Kiselman's semigroup} $K_n$ is defined by
\begin{equation*}
    K_n = \langle a_1, a_2, \dots, a_n \mid a_i^2 = a_i, a_i a_j a_i = a_j a_i a_j = a_i a_j, 1 \leq j < i \leq n \rangle.
\end{equation*}
We denote by $e$ the unit element of $K_n$.

This family of semigroups arose as a natural 
generalisation of a semigroup of operators in convex analysis introduced 
by Kiselman~\cite{kiselman}.
Fundamental structural properties of $K_n$ were established by 
Kudryavtseva and Mazorchuk~\cite{kudryavtseva}. 
The semigroup arises in graph dynamics, as discussed by Collina and D'Andrea~\cite{collina},
and exhibits rich combinatorial structure.

A more general family of monoids was introduced by Ganyushkin and Mazorchuk~\cite{ganyushkin11}.
They defined the \emph{Hecke--Kiselman monoid} $\mathrm{HK}_{\Theta}$ associated to a finite simple digraph $\Theta$.

Identities in certain families of these monoids were studied by Ashikhmin, Volkov, and Zhang~\cite{Ashikhmin15}, who solved the corresponding finite basis problem, and
further semigroup identities were investigated by Wiertel~\cite{wiertel24}.
Normal forms for two particular families were described by 
Lebed~\cite{Lebed2025}, yielding solutions to the word problem in 
those cases, while normal forms for another family were established by Aragona and D'Andrea~\cite{aragona20}.

Conditions for finiteness were analysed by 
Aragona and D'Andrea~\cite{aragona13}, who established necessary conditions on 
$\Theta$ and provided a full
classification of simple digraphs on at most four vertices for which 
$\mathrm{HK}_{\Theta}$ is finite.
The finiteness of $K_n$ was proved in~\cite[Theorem~3]{kudryavtseva}, and
D'Andrea and Stella~\cite{dandrea23} showed
that the cardinality of $K_n$ grows
double-exponentially as $n$ tends to $\infty$. 

Effective representations of Hecke--Kiselman monoids were discussed by Forsberg~\cite{forsberg2017}.
We determined all endomorphisms of $K_n$ in~\cite{andrensek}.

Representations, properties, and structure of
the corresponding Hecke--Kiselman algebras
were studied in~\cite{Mecel2019, okninski20, okninski21, wiertel22, wiertel23}.  

In~\cite[Section~10]{kudryavtseva}, deletion properties were established in $K_n$.
The present paper continues the study of these deletion properties, revealing further combinatorial consequences. 
In particular, we investigate equations involving the zero element $f$. We analyse equations of the form 
$xy = f$ under various restrictions on $y \in K_n$.
The paper relies heavily on the concept of canonical words, which was introduced in~\cite[Section~3]{kudryavtseva}.

\section{Main Results}

In Section~\ref{IDE}, we recall that $K_n$ has a zero element, which we denote by $f$.
Let $\langle a_2, a_3, \dots, a_n \rangle$ be the subsemigroup of $K_n$ generated by $\{a_2, a_3, \dots, a_n\}$.
We prove the Zero Cancellation Theorem, which states the following.
\[
\text{If } y \in \langle a_2, a_3, \dots, a_n \rangle \text{ and } xy = f, \text{ then } x = f.
\]

In contrast, the equation $x a_1 = f$ admits non-trivial solutions. We show that
\[
\bigl|\{x \in K_n \mid x a_1 = f\}\bigr| = 1 + |K_{n-1}|
\]
and give an explicit description of its solution set.
We further prove that this set is a subsemigroup of $K_n$ and derive an explicit multiplication rule.

Using ideas from the proofs of these results, we show that
\[
|K_{2n+1}| \text{ is even}, \qquad |K_{2n}| \text{ is odd}.
\]

The paper proceeds as follows.
In Section~\ref{Prelim}, we summarise the main definitions and results on $K_n$. 
In Section~\ref{ZCT}, we state and prove the Zero Cancellation Theorem, and discuss some of its consequences. 
In Section~\ref{TE}, we analyse solutions to the equation $x a_1 = f$, determine the cardinality of the solution set, give its description, and study its algebraic properties.
The paper concludes with Section~\ref{PAR}, in which we determine the parity of $|K_n|$.

\section{Preliminaries}\label{Prelim}

For a set $X$, we denote by 
$W(X)$ the set of all finite words over $X$ and we denote by $\eq$ the empty word.
The set $W(X)$ equipped with the binary operation of \emph{concatenation} becomes a monoid with unit element $\eq$.
We denote $A = \{\aq_1, \aq_2, \dots, \aq_n\}$.
Define the \emph{canonical epimorphism} $\varphi : W(A) \to K_n$ by 
\[
\varphi(\aq_{i_1} \aq_{i_2} \dots \aq_{i_k}) = a_{i_1} a_{i_2} \dots a_{i_k},
\]
for all $k \geq 1$ and $i_1, i_2, \dots, i_k \in \{1, 2, \dots, n\}$, and $\varphi(\eq) = e$.
In this paper, elements of $W(A)$ are written in boldface, for example $\uq, \vq, \wq, \aq_i$, whereas elements of $K_n$ are written in standard mathematical italic, for example $x, y, z, a_i$.

We recall the following lemma from~\cite[Lemma~1]{kudryavtseva}.
\begin{lemma}\label{lemma:shorten}
    The following statements hold.
    \begin{enumerate}[label=(\roman*), ref=\roman*]
        \item\label{lemma:shorten:i} Let $i \in \{1,2, \dots, n\}$
        and $\wq \in W(\{\aq_1, \dots, \aq_{i - 1}\})$. Then we
        have $a_i \varphi(\wq) a_i = a_i \varphi(\wq)$.
        \item\label{lemma:shorten:ii} Let $i \in \{1,2, \dots, n\}$ and $\wq \in W(\{\aq_{i + 1}, \dots, \aq_{n}\})$. Then we have
        $a_i \varphi(\wq) a_i = \varphi(\wq) a_i$.
    \end{enumerate}
\end{lemma}

We define the \emph{length} of a word as the map $l : W(A) \to \mathbb{N}_0$ given in the natural way.
In~\cite[Corollary~2]{kudryavtseva}, the authors proved that there is an upper
bound on the number of occurrences of the letters $\aq_i$ in words of shortest possible
length that represent the same element of $K_n$.
They later showed that these words correspond to \emph{canonical words} (see Section~\ref{CAN}).
\begin{proposition}\label{prop:word_bound}
    Let $x \in K_n$ and let $\wq \in W(A)$ be a word of shortest
    possible length such that $x = \varphi(\wq)$. Then we have
    \begin{enumerate}[label=(\roman*), ref = \roman*]
        \item\label{prop:word_bound:i} For $i \leq \lceil \frac{n}{2} \rceil$, the number of occurrences of $\aq_i$ in $\wq$ is less than or equal to $2^{i - 1}$.
        \item\label{prop:word_bound:ii} For $i \geq \lceil \frac{n + 1}{2} \rceil$, the number of occurrences of $\aq_i$ in $\wq$ is less than or equal to $2^{n - i}$.
    \end{enumerate}
\end{proposition}

\subsection{Canonical Words}\label{CAN}

Let $\wq \in W(A)$ be a word and write $\wq = \aq_{i_1} \aq_{i_2} \dots \aq_{i_k}$. A word of the form 
$\uq = \aq_{i_j} \aq_{i_{j+1}} \dots \aq_{i_{j + l}}$, where 
$j \geq 1$ and $j + l \leq k$, is called a 
\emph{subword} of $\wq$. 

\begin{definition}
    We call a word $\wq \in W(A)$ \emph{canonical} over $A$, or simply
    \emph{canonical}, if for every $i \in \{1, 2, \dots, n\}$ such that $\aq_i \uq \aq_i$ is a subword
    of $\wq$ for some $\uq \in W(A)$, there exist $j > i$ and $k < i$ such that $\aq_j$ and $\aq_k$ occur in $\uq$. 

    Furthermore, for a word $\wq \in W(A)$, we define the set
    $\overline{\wq} = \{\vq \in W(A) \mid \varphi(\vq) = \varphi(\wq)\} = \varphi^{-1}
    (\{\varphi(\wq)\})$.
\end{definition}
Observe that if $\vq$ is a subword of a canonical word, then 
$\vq$ is also canonical.
In~\cite[Theorem~6]{kudryavtseva}, the following theorem was proved.
\begin{theorem}\label{theorem:1}
    The following statements hold.
    \begin{enumerate}[label=(\roman*), ref=\roman*]
        \item\label{theorem:1:i} For every $\wq \in W(A)$, the set $\overline{\wq}$ contains a unique element of minimal length.
        \item\label{theorem:1:ii} A word $\vq \in \overline{\wq}$ is the unique element of minimal length if and only if $\vq$ is canonical over $A$. 
    \end{enumerate}
\end{theorem}
For a word $\wq \in W(A)$, we call the unique word of minimal length in $\overline{\wq}$ the \emph{canonical form} of $\wq$ and we denote it by $\can(\wq)$.
Since $\can(\wq)$ is the word of minimal length such that $\varphi(\wq) = \varphi(\can(\wq))$,~Proposition~\ref{prop:word_bound} may be applied to $\can(\wq)$ for any $\wq \in W(A)$.

\begin{definition}\label{definition11}
    We define the binary relation $\to$ on $W(A)$ as follows. For $\wq, \wq' \in W(A)$, we write $\wq \to \wq'$ if and only if 
    there exist $i \in \{1, 2, \dots, n\}$ and words $\wq_1, \wq_2, \sq \in W(A)$ such that $\wq = \wq_1 \aq_i \sq \aq_i \wq_2$ and
    one of the following holds:
    \begin{enumerate}
        \item $\wq' = \wq_1 \aq_i \sq \wq_2$ and $\sq \in W(\{\aq_1, \aq_2, \dots, \aq_{i-1}\})$,
        \item $\wq' = \wq_1 \sq \aq_i \wq_2$ and $\sq \in W(\{\aq_{i+1}, \aq_{i+2}, \dots, \aq_n\})$.
    \end{enumerate}
\end{definition}

We say that $\wq'$ is obtained from $\wq$ by a 
\emph{right deletion} if $\wq \to \wq'$ and the first condition in Definition~\ref{definition11} holds.
Similarly, we say that $\wq'$ is obtained from $\wq$ by a 
\emph{left deletion} if $\wq \to \wq'$ and the second condition in Definition~\ref{definition11} holds.
Furthermore, applying Lemma~\ref{lemma:shorten}~(\ref{lemma:shorten:i}) corresponds to using a right deletion, and applying Lemma~\ref{lemma:shorten}~(\ref{lemma:shorten:ii})
corresponds to using a left deletion.

For $k \geq 1$, define the binary relation $\overset{k}{\to}$ as follows. 
For words $\wq, \wq' \in W(A)$, we write $\wq \overset{k}{\to} \wq'$ if and only if there exist words 
$\vq_1, \dots, \vq_{k-1} \in W(A)$ such that $\wq \to \vq_1 \to \dots \to \vq_{k-1} \to \wq'$. 
We also define $\overset{0}{\to}$ as $\wq \overset{0}{\to} \wq'$ if and only if $\wq = \wq'$. 
Observe that the binary relations $\overset{1}{\to}$ and $\to$ coincide.
Denote by $\overset{*}{\to} = \bigcup_{k \geq 0} \overset{k}{\to}$ the reflexive-transitive closure of $\to$ on $W(A)$. 

The following proposition collects basic facts about canonical forms and the binary relation $\to$, which are used throughout the paper without explicit reference.
\begin{proposition}\label{prop:basic}
    Let $\vq, \wq \in W(A)$.
    The following statements hold.
    \begin{enumerate}[label=(\roman*), ref=\roman*]
        \item\label{prop:basic:i} If $\varphi(\vq) = \varphi(\wq)$, then $\can(\vq) = \can(\wq)$.
        \item\label{prop:basic:ii} The word $\wq$ is canonical if and only if $\wq = \can(\wq)$.
        \item\label{prop:basic:iii} If $\vq \overset{*}{\to} \wq$, then $\varphi(\vq) = \varphi(\wq)$.
        \item\label{prop:basic:iv} We have $\wq \overset{*}{\to} \can(\wq)$.
    \end{enumerate}
\end{proposition}

\begin{proof}
    (\ref{prop:basic:i})~If $\varphi(\vq) = \varphi(\wq)$, then $\overline{\vq} = \overline{\wq}$, and hence $\can(\vq) = \can(\wq)$.
    (\ref{prop:basic:ii})~Since $\wq \in \overline{\wq}$, this follows from Theorem~\ref{theorem:1}~(\ref{theorem:1:ii}).
    (\ref{prop:basic:iii})~If $\vq \to \wq$, then $\varphi(\vq) = \varphi(\wq)$ by Lemma~\ref{lemma:shorten}. Hence, if $\vq \overset{*}{\to} \wq$, then $\varphi(\vq) = \varphi(\wq)$.
    (\ref{prop:basic:iv})~This follows from the proof of~\cite[Theorem~6]{kudryavtseva}; $\can(\wq)$ is the \emph{normal form} of $\wq$ with respect to the binary relation $\to$ on $W(A)$.
    See~\cite[Chapter~2.1]{baader} for a general reference on term rewriting.
\end{proof}

For $x \in K_n$, we define its \emph{canonical form}, denoted by $\zeta(x)$, to be $\can(\wq)$, where $\wq \in W(A)$ satisfies $\varphi(\wq) = x$. 
This is well defined, since Proposition~\ref{prop:basic}~(\ref{prop:basic:i}) implies that all such words $\wq$ have the same canonical form.
We then have $x = \varphi(\zeta(x))$ for every $x \in K_n$.
\begin{lemma}\label{lemma:canonical_unique}
    Let $x, y \in K_n$.
    If $\zeta(x) = \zeta(y)$, then $x = y$.
\end{lemma}

\begin{proof}
    Write $x = \varphi(\wq)$ and $y = \varphi(\vq)$ for some $\wq, \vq \in W(A)$.
    Since $\zeta(x) = \zeta(y)$, we have $\can(\wq) = \can(\vq)$.
    Then we have $\can(\wq) \in \overline{\wq}$ and $\can(\wq) = \can(\vq) \in \overline{\vq}$.
    This implies $\varphi(\wq) = \varphi(\can(\wq)) = \varphi(\vq)$, which yields $x = y$.
\end{proof}

In~\cite[Proposition~29]{kudryavtseva}, the following proposition was proved. 
\begin{proposition}\label{p:can_injective}
    Let $\uq, \vq, \wq \in W(\{\aq_2, \aq_3, \dots, \aq_n\})$ be canonical words.
    Assume that $\wq \aq_1 \uq$ and $\wq \aq_1 \vq$ are both canonical. 
    If $\uq \neq \vq$, 
    then 
    $\varphi(\wq \uq) \neq \varphi(\wq \vq)$.
\end{proposition}
This proposition implies that if 
$\uq, \vq, \wq \in W(\{\aq_2, \aq_3, \dots, \aq_n\})$ are canonical words such that 
$\wq \aq_1 \uq$ and $\wq \aq_1 \vq$ are both canonical, then the equality 
$\varphi(\wq \uq) = \varphi(\wq \vq)$ implies that $\uq = \vq$.

\subsection{Idempotents}\label{IDE}

Let $X \subseteq \{1, 2, \dots, n\}$. If $X = \emptyset$, define $e_X = e$, the unit element in $K_n$, and $\eq_X = \eq$, the empty word.
Otherwise, write $X = \{i_1, i_2, \dots, i_k\}$ with $i_1 > i_2 > \dots > i_k$ and set
\[
e_X = a_{i_1} a_{i_2} \dots a_{i_k} \quad \text{and} \quad \eq_X = \aq_{i_1} \aq_{i_2} \dots \aq_{i_k}.
\]
Then we have $\varphi(\eq_X) = e_X$ and the word $\eq_X$ is canonical for every $X \subseteq \{1, 2, \dots, n\}$.

In~\cite[Remark~16]{kudryavtseva}, it was established that $e_{\{1, 2, \dots, n\}}$ is the zero
element of $K_n$. We denote 
$e_{\{1, 2, \dots, n\}}$ by $f$, and $\eq_{\{1, 2, \dots, n\}}$ by $\fq$.

\begin{remark}\label{rm1}
    For every $\wq \in W(A)$, we have $\fq \wq \overset{*}{\to} \fq$.
    Therefore, the word $\fq \wq$ is canonical if and only if $\wq$ is the empty word.
\end{remark}

Throughout the paper, we often use the trivial
equalities 
\[
e_{\{2, 3, \dots, n\}} a_1 = f \quad \text{and} \quad \eq_{\{2, 3, \dots, n\}} \aq_1 = \fq.
\]

In~\cite[Proposition~11]{kudryavtseva}, the authors showed that $e_X$ is an idempotent in 
$K_n$ for every $X \subseteq \{1, 2, \dots, n\}$.
Furthermore, they showed that $e_X \neq e_Y$ if $X \neq Y$,
and proved that
every idempotent in $K_n$ is of the form $e_X$ for some $X \subseteq \{1, 2, \dots, n\}$.
As a result, $K_n$ contains exactly $2^n$ idempotents.

Define the \emph{content} map on $K_n$ as $c : K_n \to \mathbbm{2}^{\{1, 2, \dots, n\}}$, where $c(x)$ is the set of
all $i \in \{1, 2, \dots, n\}$ such that $\aq_i$ occurs 
in the canonical form of $x$. 
In~\cite[Lemma~10]{kudryavtseva}, it was shown that the following holds.
\begin{proposition}
  The map $c$ is a semigroup epimorphism from $K_n$ onto the semigroup $(\mathbbm{2}^{\{1, 2, \dots, n\}}, \cup)$.
\end{proposition} 

In~\cite[Proposition~20]{kudryavtseva}, the following was proved.
\begin{proposition}
    The only automorphism of $K_n$ is the identity. The map $a_i \mapsto a_{n-i+1}$ extends uniquely to an 
    antiautomorphism of $K_n$. This is the only antiautomorphism of $K_n$.
\end{proposition}

We denote the unique antiautomorphism of $K_n$ by $\tau$.

\begin{remark}\label{tau_rm}
    Observe that $\tau(f) = \tau(a_n a_{n-1} \dots a_2 a_1) = 
    \tau(a_1) \tau(a_2) \dots \tau(a_{n-1}) \tau(a_n) = a_n a_{n-1} \dots a_2 a_1 = f$. Hence 
    $\tau$ fixes $f$.
\end{remark}

\section{Zero Cancellation Theorem}\label{ZCT}

In this section, we prove the Zero Cancellation Theorem, which we now state.
\begin{theorem}[Zero Cancellation Theorem]\label{ct1}
    Let $x, y \in K_n$ satisfy $xy = f$. 
    \begin{enumerate}[label=(\roman*), ref = \roman*]
        \item\label{ct1:i} If $c(y) \subseteq \{2, 3, \dots, n\}$, then $x = f$.
        \item\label{ct1:ii} If $c(x) \subseteq \{1, 2, \dots, n-1\}$, then $y = f$.
    \end{enumerate}
\end{theorem}

The proof proceeds as follows. We first recall the notion 
of a quasi-subword. We then establish several technical results 
concerning canonical forms, exploiting the key role of the letter $\aq_1$, and
use these to prove the Zero Cancellation Theorem.

\begin{definition}
    Let $\wq \in W(A)$ be a word. Write $\wq = \aq_{i_1} \aq_{i_2} \dots \aq_{i_k}$.
    A word of the form $\vq = \aq_{i_{j_1}} \aq_{i_{j_2}} \dots \aq_{i_{j_l}}$, where $1 \leq j_1 < j_2 < \dots < j_l \leq k$, is called a \emph{quasi-subword} of $\wq$.
    In this case, we write $\vq \leq \wq$.
\end{definition}

It follows immediately from the definition that if
$\uq \leq \vq$ and $\vq \leq \wq$, then 
$\uq \leq \wq$.
Moreover, we have $\wq \leq \wq$ for every word $\wq \in W(A)$.

\begin{lemma}\label{can_quasi}
    For every $\wq \in W(A)$, we have $\can(\wq) \leq \wq$.
\end{lemma}
\begin{proof}
    If $\wq = \can(\wq)$, the statement follows from $\wq \leq \wq$.
    Otherwise, there exist $k \geq 1$ and words $\vq_1, \dots, \vq_{k-1} \in W(A)$ such that 
    $\wq \to \vq_1 \to \dots \to \vq_{k-1} \to \can(\wq)$.
    Since $\leq$ is transitive, it suffices to show that for any 
    $\uq, \uq' \in W(A)$, if 
    $\uq \to \uq'$, then $\uq' \leq \uq$.
    This follows immediately from the definition of $\to$.
\end{proof}

\begin{remark}\label{remark:a_i_occur}
    Let $\wq \in W(A)$. If the letter $\aq_i$ occurs in $\can(\wq)$, 
    then $\aq_i \leq \can(\wq) \leq \wq$, which implies
    that $\aq_i \leq \wq$.
    Therefore, $\aq_i$ occurs in $\wq$.
\end{remark}

The following lemma describes the canonical form of 
words of the form $\wq \aq_1 \uq$. The essence of this result is used implicitly
in the proof of~\cite[Proposition~28]{kudryavtseva}. For the sake of
completeness, we state and prove it here.
\begin{lemma}\label{lc0}
    Let $\wq, \uq \in W(\{\aq_2, \aq_3, \dots, \aq_n\})$ and assume that $\wq$ is canonical.
    Then
    \begin{equation*}
        \can(\wq \aq_1 \uq) = \wq \aq_1 \uq^*,
    \end{equation*}
    for some $\uq^* \leq \uq$.
\end{lemma}

\begin{proof}
    If $\wq \aq_1 \uq$ is already canonical, the conclusion follows from $\uq \leq \uq$.
    Otherwise, there exist $k \geq 1$ and words 
    $\vq_1, \dots, \vq_{k-1} \in W(A)$ such that
    \begin{equation*}
        \wq \aq_1 \uq \to \vq_1 \to \dots \to \vq_{k-1} \to \can(\wq \aq_1 \uq). 
    \end{equation*}
    Since $\leq$ is transitive, it suffices to show that whenever 
    $\vq' = \wq \aq_1 \uq'$ with $\uq' \leq \uq$ and 
    $\vq' \to \vq''$, there exists $\uq'' \leq \uq'$ such that $\vq'' = \wq \aq_1 \uq''$.

    Assume $\vq' = \wq \aq_1 \uq'$ with $\uq' \leq \uq$ and $\vq' \to \vq''$.
    By the definition of $\to$, there exist $i \in \{1, 2, \dots, n\}$ and words
    $\wq_1, \wq_2, \sq \in W(A)$ such that 
    $\vq' = \wq_1 \aq_i \sq \aq_i \wq_2$ and one of the following holds:
    \begin{enumerate}
        \item $\vq'' = \wq_1 \aq_i \sq \wq_2$ and $\sq \in W(\{\aq_1, \aq_2, \dots, \aq_{i-1}\})$,
        \item $\vq'' = \wq_1 \sq \aq_i \wq_2$ and $\sq \in W(\{\aq_{i+1}, \aq_{i+2}, \dots, \aq_n\})$.
    \end{enumerate}
    Since $\uq' \leq \uq$ and $\uq \in W(\{\aq_2, \aq_3, \dots, \aq_n\})$, the letter $\aq_1$ does not occur in $\uq'$.
    Since $\aq_1$ also does not occur in $\wq$, it occurs exactly once in $\wq \aq_1 \uq' = \vq'$, and hence $i > 1$.
    Both occurrences of $\aq_i$ from the definition of $\to$ cannot lie in $\wq$, since this would produce the word
    $\aq_i \sq \aq_i$ with $\sq \in W(\{\aq_1, \dots, \aq_{i-1}\})$ or $\sq \in W(\{\aq_{i+1}, \dots, \aq_n\})$
    as a subword of $\wq$, which would contradict the assumption that $\wq$ is canonical.
    Therefore, at least one of the two occurrences of $\aq_i$ from the definition of $\to$ lies in $\uq'$. 
    
    If both occurrences of $\aq_i$ from the definition of $\to$ lie in $\uq'$, the reduction $\vq' \to \vq''$ deletes
    one occurrence of $\aq_i$ from $\uq'$, while the prefix $\wq \aq_1$ remains unchanged. 
    Let $\uq''$ be the word obtained from $\uq'$ by deleting this occurrence of $\aq_i$.
    Then we have $\vq'' = \wq \aq_1 \uq''$ and $\uq'' \leq \uq'$.

    If only one occurrence of $\aq_i$ from the definition of $\to$ lies in $\uq'$, then it must be the right occurrence, and the left occurrence of $\aq_i$ lies in $\wq$, since $i > 1$.
    This implies that $\aq_1$ occurs in the subword $\sq$.
    Therefore, we cannot have $\sq \in W(\{\aq_{i+1}, \dots, \aq_n\})$, and hence we are in the first case of the definition of $\to$.
    Thus, in the reduction $\vq' \to \vq''$, we use a right deletion to delete the occurrence of $\aq_i$ in $\uq'$,
    while the prefix $\wq \aq_1$ remains unchanged.
    Let $\uq''$ be the word obtained by deleting the mentioned occurrence of $\aq_i$ from $\uq'$.
    We then have $\vq'' = \wq \aq_1 \uq''$ and $\uq'' \leq \uq'$.
    This concludes the proof.
\end{proof}

The following theorem provides a key step in proving the main result of this section.
It asserts that if $\wq$ is canonical and contains $\aq_1$, and $\uq$ does not contain $\aq_1$, then the prefix preceding $\aq_1$ in $\wq$ coincides with the prefix
preceding $\aq_1$ in $\can(\wq \uq)$.
\begin{theorem}\label{main}
    Let $\wq \in W(A)$ be canonical and let $\uq, \vq_1, \vq_2 \in W(\{\aq_2, \aq_3, \dots, \aq_n\})$. If we have 
    $\can(\wq \uq) = \vq_1 \aq_1 \vq_2$,
    then $\wq = \vq_1 \aq_1 \vq'$ for some $\vq' \in W(\{\aq_2, \aq_3, \dots, \aq_n\})$.
\end{theorem}

\begin{proof}
    Since $\can(\wq \uq) = \vq_1 \aq_1 \vq_2$, Lemma~\ref{can_quasi} implies $\vq_1 \aq_1 \vq_2 \leq \wq \uq$.
    Since $\aq_1 \leq \vq_1 \aq_1 \vq_2 \leq \wq \uq$, Remark~\ref{remark:a_i_occur} further implies that the letter $\aq_1$ occurs in $\wq \uq$.
    As $\uq \in W(\{\aq_2, \aq_3, \dots, \aq_n\})$, the letter $\aq_1$ occurs at least once in $\wq$.
    Since $\wq$ is canonical, the letter $\aq_1$ occurs 
    exactly once in $\wq$ by Proposition~\ref{prop:word_bound}~(\ref{prop:word_bound:i}).
     
    Since $\aq_1$ occurs exactly once in $\wq$, we can write $\wq$ as
    \begin{equation*}
        \wq = \wq_1 \aq_1 \wq_2,
    \end{equation*}
    for some $\wq_1, \wq_2 \in W(\{\aq_2, \dots, \aq_n\})$. Hence, the word $\wq \uq$ is of the form 
    $\wq \uq = \wq_1 \aq_1 (\wq_2 \uq)$. Since $\wq_1 \in W(\{\aq_2, \dots, \aq_n\})$
    is canonical, as it is a subword of the canonical word $\wq$, and since $\wq_2 \uq \in W(\{\aq_2, \dots, \aq_n\})$,
    we can use Lemma~\ref{lc0} to conclude that $\can(\wq \uq) = \can(\wq_1 \aq_1 \wq_2 \uq) = \wq_1 \aq_1 \uq^*$ for some $\uq^* \leq \wq_2 \uq$. 
    Thus, we have 
    \begin{equation*}
        \wq_1 \aq_1 \uq^* = \can(\wq \uq) = \vq_1 \aq_1 \vq_2.
    \end{equation*}
    Since $\wq_1, \vq_1 \in W(\{\aq_2, \aq_3, \dots, \aq_n\})$, we must have $\wq_1 = \vq_1$
    and $\uq^* = \vq_2$. Hence, $\wq = \wq_1 \aq_1 \wq_2 = \vq_1 \aq_1 \wq_2$. Setting $\vq' = \wq_2$ yields 
    the desired conclusion.
\end{proof}

\begin{theorem}\label{ct0}
    Let $\wq \in W(A)$ be canonical and let $\uq \in W(\{\aq_2, \aq_3, \dots, \aq_n\})$. 
    If $\can(\wq \uq) = \fq$, then $\wq = \fq$.
\end{theorem}

\begin{proof}
    We apply Theorem~\ref{main} with $\vq_1 = \eq_{\{2, 3, \dots, n\}}$ and 
    $\vq_2$ the empty word, and conclude that
    \begin{equation*}
        \wq = \eq_{\{2, 3, \dots, n\}} \aq_1 \vq' = \fq \vq',
    \end{equation*}
    for some $\vq' \in W(\{\aq_2, \aq_3, \dots, \aq_n\})$.
    
    Using Remark~\ref{rm1}, we see that 
    $\wq = \fq \vq'$ is canonical if and only if $\vq'$ is the empty word.
    The canonicity assumption on $\wq$ yields that
    $\vq'$ is indeed the empty word. Hence, we have $\wq = \fq$, which
    concludes the proof of the theorem.
\end{proof}

Now the proof of the Zero Cancellation Theorem follows easily.
\begin{proof}[Proof of Theorem~\ref{ct1}]
    
    (\ref{ct1:i}) Write $x = \varphi(\wq)$ and $y = \varphi(\wq')$,
    where $\wq$ and $\wq'$ are the canonical forms of $x$ and $y$, respectively. 
    Since $c(y) \subseteq \{2, 3, \dots, n\}$, we have 
    $\wq' \in W(\{\aq_2, \aq_3, \dots, \aq_n\})$.
    Since $\varphi(\wq \wq') = xy = f = \varphi(\fq)$ and
    $\fq$ is a canonical word, it follows that
    $\can(\wq \wq') = \fq$. As 
    $\wq$ is canonical and 
    $\wq' \in W(\{\aq_2, \aq_3, \dots, \aq_n\})$, we can apply
    Theorem~\ref{ct0} and conclude that 
    $\wq = \fq$, and hence $x=f$.

    (\ref{ct1:ii}) Define $x' = \tau(x)$ and $y' = \tau(y)$. 
    Since $c(x) \subseteq \{1, 2, \dots, n-1\}$, we have 
    $c(x') \subseteq \{2, 3, \dots, n\}$. We also have
    \begin{align*}
        y' x' = \tau(y) \tau(x) = \tau(xy) = \tau(f) = f,
    \end{align*}
    where the last equality holds by Remark~\ref{tau_rm}.
    Therefore, by~(\ref{ct1:i}), we conclude that $\tau(y) = y' = f$. 
    Since $\tau$ is injective and 
    $\tau(f) = f$, it follows that $y = f$.
\end{proof}

\begin{corollary}\label{ct2}
    Let $x \in K_n$ and let
    $k \in \{2, 3, \dots, n\}$ and $l \in \{1, 2, \dots, n-1\}$.
    If $x a_k = f$ or $a_l x = f$, then $x=f$. 
\end{corollary}

\begin{proof}
    If $x a_k = f$, then $c(a_k) \subseteq \{2, 3, \dots, n\}$ and Theorem~\ref{ct1}~(\ref{ct1:i}) imply $x = f$.
    If $a_l x = f$, then $c(a_l) \subseteq \{1, 2, \dots, n-1\}$ and Theorem~\ref{ct1}~(\ref{ct1:ii}) imply $x = f$.
\end{proof}

\begin{corollary}
    Let $x, y, z \in K_n$ and assume that
    $c(x) \subseteq \{1, 2, \dots, n-1\}$ and $c(z) \subseteq \{2, 3, \dots, n\}$. 
    If $x y z = f$, then $y = f$.
\end{corollary}

\begin{proof}
    Since $(xy)z = f$ and 
    $c(z) \subseteq \{2, 3, \dots, n\}$,
    Theorem~\ref{ct1}~(\ref{ct1:i}) implies that $xy = f$. As $c(x) \subseteq \{1, 2, \dots, n-1\}$,
    Theorem~\ref{ct1}~(\ref{ct1:ii}) yields
    $y = f$.
\end{proof}

We call Theorem~\ref{ct1} the Zero Cancellation Theorem since we can rewrite 
the equality $xy = f$ as $xy = fy = xf$, and Theorem~\ref{ct1} states that
we may cancel out $x$ or $y$ under the mentioned content restrictions.
The term zero refers to
the zero element $f$ of $K_n$.

We conclude this section with an interesting characterisation of the zero 
element $f$ of $K_n$.
\begin{corollary}
    Let $x \in K_n$. Then $x = f$ if and only if 
    $xa_k = f$ for some $k \in \{2, 3, \dots, n\}$.
\end{corollary}

\begin{proof}
    If $x = f$, we have 
    $x a_k = f$ for all $k \in \{1, 2, \dots, n\}$. 
    The converse follows immediately from Corollary~\ref{ct2}.
\end{proof}

\section{The Equation $x a_1 = f$}\label{TE}

In this section, we study the equation $x a_1 = f$ in $K_n$.
Applying $\tau$ to this equation yields the equivalent problem of determining all $x \in K_n$ such that $a_n x = f$.

Corollary~\ref{ct2} states that for $k \in \{2, 3, \dots, n\}$, the equation 
$x a_k = f$ has the unique solution $x = f$.
However, Corollary~\ref{ct2} does not hold for $k = 1$.
Indeed, for $x = e_{\{2, 3, \dots, n\}} \neq f$, we also have $x a_1 = f$. 
Hence, the set of solutions to the equation $x a_1 = f$ is not the
singleton $\{f\}$. In general, we have the following result.
\begin{proposition}
    The equation $x y = f$ admits only the trivial solution $x = f$ if and only if $c(y) \subseteq \{2, 3, \dots, n\}$.
\end{proposition}

\begin{proof}
    By Theorem~\ref{ct1}~(\ref{ct1:i}), if $c(y) \subseteq \{2, 3, \dots, n\}$, then $x = f$ is the unique solution to $xy = f$.

    Conversely, if $1 \in c(y)$, Proposition~\ref{prop:word_bound}~(\ref{prop:word_bound:i}) implies that the canonical form of $y$ is of the form
    $\wq \aq_1 \uq$, where $\wq, \uq \in W(\{\aq_2, \aq_3, \dots,\aq_n\})$.
    Setting $z = \varphi(\wq)$ and $y' = \varphi(\uq)$, we obtain
    \begin{equation*}
        y = z a_1 y',
    \end{equation*}
    where $z, y' \in K_n$ and $c(z), c(y') \subseteq \{2, 3, \dots, n\}$.
    By Lemma~\ref{lemma:shorten}~(\ref{lemma:shorten:i}), we have $e_{\{2, 3, \dots, n\}} z = e_{\{2, 3, \dots, n\}}$, and hence 
    \begin{equation*}
        e_{\{2, 3, \dots, n\}} y = e_{\{2, 3, \dots, n\}} z a_1 y' = e_{\{2, 3, \dots, n\}} a_1 y' = f y' = f.
    \end{equation*}
    Thus, $x = e_{\{2, 3, \dots, n\}}$ is a non-trivial solution to $xy = f$.
\end{proof}

We focus on the case when $y = a_1$. We denote by
\begin{equation*}
    R = \{x \in K_n \mid x a_1 = f\}
\end{equation*}
the set of all solutions to $x a_1 = f$ and
we also denote by 
\begin{equation*}
    K_n^1 = \{x \in K_n \mid 1 \in c(x)\}
\end{equation*}
the set of all $x \in K_n$ whose canonical form contains the letter $\aq_1$.
If $x, y \in K_n^1$, then $1 \in c(x) \cup c(y) = c(xy)$, which shows $xy \in K_n^1$.
Therefore, $K_n^1$ is a subsemigroup of $K_n$.

For $i_1, i_2, \dots, i_k \in \{1, 2, \dots, n\}$, we denote by
$\langle a_{i_1}, a_{i_2}, \dots, a_{i_k} \rangle$ the subsemigroup of $K_n$ generated by $\{a_{i_1}, a_{i_2}, \dots, a_{i_k}\}$.
If $i_1, i_2, \dots, i_k$ are pairwise distinct, then $\langle a_{i_1}, a_{i_2}, \dots, a_{i_k} \rangle$ is isomorphic to $K_k$.
The subsemigroups $K_n^1$ and $\langle a_2, a_3, \dots, a_n \rangle$ are disjoint and their union is $K_n$.
Hence, defining
\begin{equation*}
    T = \{x \in K_n^1 \mid x a_1 = f\},
\end{equation*}
we can write $R$ as
\begin{equation*}
    R = \{x \in \langle a_2, a_3, \dots, a_n \rangle \mid x a_1 = f\} \cup T
\end{equation*}
and this union is disjoint.

The following lemma determines solutions to $x a_1 = f$ within $\langle a_2, a_3, \dots, a_n \rangle$.
\begin{lemma}\label{eqlemma}
    Let $x \in \langle a_2, a_3, \dots, a_n \rangle$ satisfy $x a_1 = f$.
    Then $x = e_{\{2, 3, \dots, n\}}$.
\end{lemma}

\begin{proof}
    We write $x = \varphi(\wq)$, 
    where $\wq \in W(\{\aq_2, \aq_3, \dots, \aq_n\})$ is the canonical form of $x$. 
    As $x a_1 = f$, we have $\varphi(\wq \aq_1) = x a_1 = f = \varphi(\fq)$, and since $\fq$ is canonical, we obtain $\can(\wq \aq_1) = \fq$.
    Since $\wq$ is canonical and $\aq_1$ does not occur in $\wq$, the word $\wq \aq_1$ is already canonical.
    Hence $\wq \aq_1 = \fq$.
    It follows that $\wq = \eq_{\{2, 3, \dots, n\}}$, which implies $x = e_{\{2, 3, \dots, n\}}$.
\end{proof}

By Lemma~\ref{eqlemma}, the set $R$ is equal to
\begin{equation}\label{R_decomp}
    R = \{e_{\{2, 3, \dots, n\}}\} \cup T
\end{equation}
and this union is disjoint.
We now turn to the study of the set $T$.
Observe that if $x \in K_n^1$ and $\wq$ is the canonical form of $x$, then by Proposition~\ref{prop:word_bound}~(\ref{prop:word_bound:i}), the letter 
$\aq_1$ occurs exactly once in $\wq$. This allows us to define the following map.
\begin{definition}\label{def:pi}
    Define the map $\pi : K_n^1 \to \langle a_2, a_3, \dots, a_n \rangle$ as follows.
    For $x \in K_n^1$, let $\wq \aq_1 \uq$ be its canonical form.
    Then set $\pi(x) = \varphi(\wq)$.
\end{definition}

The map $\pi$ is not a semigroup homomorphism. Indeed, let 
$n \geq 2$ and let $x = a_1 a_2$ and 
$y = a_1$. Then we have $x, y \in K_n^1$ and
$xy = a_2 a_1$. We have 
$\pi(x) = e$ and $\pi(y) = e$. However,
$\pi(xy) = a_2 \neq \pi(x) \pi(y)$.

\begin{proposition}\label{prop:res}
    We have that $\pi(T) = \langle a_2, a_3, \dots, a_n \rangle$.
\end{proposition}

\begin{proof}
    Let $x \in \langle a_2, a_3, \dots, a_n \rangle$ and write 
    $x = \varphi(\wq)$, where $\wq \in W(\{\aq_2, \aq_3, \dots, \aq_n\})$ is the canonical form of $x$.
    We define $\vq \in W(A)$ and $z \in K_n$ as 
    \begin{equation*}
        \vq = \wq \aq_1 \eq_{\{2, 3, \dots, n\}} \quad \text{and} \quad z = \varphi(\vq).
    \end{equation*}
    We claim that $z \in T$ and $x = \pi(z)$.

    Since $\wq \in W(\{\aq_2, \aq_3, \dots, \aq_n\})$ is canonical and 
    $\eq_{\{2, 3, \dots, n\}} \in W(\{\aq_2, \aq_3, \dots, \aq_n\})$, by Lemma~\ref{lc0}, we have 
    \begin{equation}\label{323}
        \can(\vq) = \can(\wq \aq_1 \eq_{\{2, 3, \dots, n\}}) = \wq \aq_1 \uq^*, 
    \end{equation}
    for some $\uq^* \leq \eq_{\{2, 3, \dots, n\}}$.
    This implies $z \in K_n^1$.
    Moreover, we have
    \begin{equation*}
        z a_1 = \varphi(\vq) a_1 = \varphi(\wq \aq_1 \eq_{\{2, 3, \dots, n\}}) a_1 = \varphi(\wq) a_1 e_{\{2, 3, \dots, n\}} a_1 
        = \varphi(\wq) a_1 f = f,    
    \end{equation*} 
    showing that $z \in T$.
    By the definition of $\pi$ and~\eqref{323}, we obtain $\pi(z) = \varphi(\wq) = x$. 
    This proves that $x \in \pi(T)$, and hence $\langle a_2, a_3, \dots, a_n \rangle \subseteq \pi(T)$.
    The converse inclusion is immediate from the definition of $\pi$.
    This concludes the proof.
\end{proof}

\begin{proposition}\label{prop:restrict}
    The restriction $\pi|_T : T \to \langle a_2, a_3, \dots, a_n \rangle$ is a bijective map.
\end{proposition}

\begin{proof}
    By Proposition~\ref{prop:res}, we have that $\pi(T) = \langle a_2, a_3, \dots, a_n \rangle$, and hence surjectivity follows.
    We now show that 
    $\pi|_T$ is an injective map. 

    Let $x_1, x_2 \in T$ and assume $\pi(x_1) = \pi(x_2)$. 
    Since $T \subseteq K_n^1$, we can write 
    \[
    x_1 = \varphi(\wq_1 \aq_1 \uq_1) \quad \text{and} \quad x_2 = \varphi(\wq_2 \aq_1 \uq_2),
    \]
    where $\wq_1, \uq_1, \wq_2, \uq_2 \in W(\{\aq_2, \aq_3, \dots, \aq_n\})$ are canonical words, and $\wq_1 \aq_1 \uq_1$ and 
    $\wq_2 \aq_1 \uq_2$ are the canonical forms of $x_1$ and $x_2$, respectively. 
    Then we have
    \begin{equation*}
        \pi(x_1) = \varphi(\wq_1) \quad \text{and} \quad \pi(x_2) = \varphi(\wq_2).
    \end{equation*} 
    Hence we have $\varphi(\wq_1) = \varphi(\wq_2)$.
    Since $\wq_1$ and $\wq_2$ are both canonical, we obtain $\wq_1 = \wq_2$.
    We now denote $\wq = \wq_1$.
    Since $x_1, x_2 \in T$ and $\uq_1, \uq_2 \in W(\{\aq_2, \aq_3, \dots, \aq_n\})$, we use 
    Lemma~\ref{lemma:shorten}~(\ref{lemma:shorten:ii}) to obtain
    \begin{align}
        x_1 a_1 &= \varphi(\wq \aq_1 \uq_1) a_1 = \varphi(\wq) a_1 \varphi(\uq_1) a_1 = \varphi(\wq) \varphi(\uq_1) a_1 = \varphi(\wq \uq_1) a_1 = f, \label{o1} \\
        x_2 a_1 &= \varphi(\wq \aq_1 \uq_2) a_1 = \varphi(\wq) a_1 \varphi(\uq_2) a_1 = \varphi(\wq) \varphi(\uq_2) a_1 = \varphi(\wq \uq_2) a_1 = f. \label{o2}
    \end{align}
    Since $\wq \uq_1, \wq \uq_2 \in W(\{\aq_2, \aq_3, \dots, \aq_n\})$, 
    we have $\varphi(\wq \uq_1), \varphi(\wq \uq_2) \in \langle a_2, a_3, \dots, a_n \rangle$.
    By Lemma~\ref{eqlemma},~\eqref{o1}, and~\eqref{o2}, it follows that 
    \begin{align*}
        \varphi(\wq \uq_1) &= e_{\{2, 3, \dots, n\}}, \\
        \varphi(\wq \uq_2) &= e_{\{2, 3, \dots, n\}}.
    \end{align*}
    Therefore, $\varphi(\wq \uq_1) = \varphi(\wq \uq_2)$.
    Since $\wq, \uq_1, \uq_2 \in W(\{\aq_2, \aq_3, \dots, \aq_n\})$ are canonical, and $\wq \aq_1 \uq_1$ and $\wq \aq_1 \uq_2$ are also canonical,
    by Proposition~\ref{p:can_injective}, we conclude that $\uq_1 = \uq_2$.
    Hence we have
    \begin{align*} 
        \wq_1 = \wq_2 \quad \text{and} \quad \uq_1 = \uq_2. 
    \end{align*}
    It follows that $x_1 = x_2$. 
    This implies that $\pi|_T$ is injective and thus the proof is complete.
\end{proof}

The following corollary determines the cardinality of the set $T$.
\begin{corollary}\label{bc1}
    We have that $|T| = |K_{n-1}|$, and
    consequently 
    $|R| = 1 + |K_{n-1}|$.
\end{corollary}

\begin{proof}
    By Proposition~\ref{prop:restrict}, the map 
    $\pi|_T : T \to \langle a_2, a_3, \dots, a_n \rangle$ is bijective.
    The map $\phi : \{a_2, a_3, \dots, a_n\} \to K_{n-1}$ defined 
    by $\phi(a_i) = a_{i-1}$ can be extended to a semigroup isomorphism
    $\widetilde{\phi}$ between $\langle a_2, a_3, \dots, a_n \rangle$ and $K_{n-1}$.
    Therefore, the map 
    $\widetilde{\phi} \circ \pi|_T : T \to K_{n-1}$ is a bijection, and hence $|T| = |K_{n-1}|$.
    By~\eqref{R_decomp}, we have $|R| = 1 + |T| = 1 + |K_{n-1}|$, which completes the proof.
\end{proof}

In what follows, we interpret the set $\{i, i+1, \dots, j\}$ for $j < i$ as the empty set.
Since $x e_{\{1, 2, \dots, n\}} = xf = f$ for any $x \in K_n$, the set 
\begin{equation*}
    \{ i \in \{0, 1, \dots, n\} \mid x e_{\{1, 2, \dots, i\}} = f \}
\end{equation*}
is nonempty for any $x \in K_n$. 
This ensures that the following map is well defined. 
\begin{definition}\label{definition_m}
    We define the map $m : K_n \to \{0, 1, \dots, n\}$ by 
    \begin{equation*}
            m(x) = \min \{ i \in \{0, 1, \dots, n\} \mid x e_{\{1, 2, \dots, i\}} = f \}.
    \end{equation*}
\end{definition}
Observe that $m(x) = 0$ if and only if $x = f$.

\begin{theorem}\label{Theorem_M}
    Let $x \in \langle a_2, a_3, \dots, a_n \rangle$.
    Then we have 
    \[
    \zeta(x a_1 e_{\{2, 3, \dots, m(x)\}}) = \zeta(x) \aq_1 \eq_{\{2, 3, \dots, m(x)\}}.
    \]
\end{theorem}

\begin{proof}
    Let $x \in \langle a_2, a_3, \dots, a_n \rangle$.
    Then $x \neq f$, and hence $m(x) \geq 1$.
    Let $k = m(x)$.
    We first show that $\zeta(x) \aq_1 \eq_{\{2, 3, \dots, k\}}$ is canonical.

    If $k = 1$, then by the definition of $k$, we have $x e_{\{1, 2, \dots, 1\}} = x a_1 = f$, and hence, by Lemma~\ref{eqlemma}, $x = e_{\{2, 3, \dots, n\}}$. Then
    \begin{equation*}
        \zeta(x) \aq_1 \eq_{\{2, 3, \dots, k\}} = \eq_{\{2, 3, \dots, n\}} \aq_1 \eq_{\{2, 3, \dots, 1\}} = \eq_{\{2, 3, \dots, n\}} \aq_1 = \fq.
    \end{equation*}
    Hence $\zeta(x) \aq_1 \eq_{\{2, 3, \dots, k\}}$ is canonical.

    Now assume $k > 1$ and suppose that $\zeta(x) \aq_1 \eq_{\{2, 3, \dots, k\}}$ is not canonical.
    Since $\zeta(x)$ and $\eq_{\{2, 3, \dots, k\}}$ lie in $W(\{\aq_2, \aq_3, \dots, \aq_n\})$ and are both canonical, there exists a subword of $\zeta(x) \aq_1 \eq_{\{2, 3, \dots, k\}}$ 
    of the form $\aq_i \sq \aq_i$, where $i > 1$, the left $\aq_i$ occurs in $\zeta(x)$, the right $\aq_i$ occurs in $\eq_{\{2, 3, \dots, k\}}$, and 
    \begin{equation*}
        \sq \in W(\{\aq_1, \aq_2, \dots, \aq_{i-1}\}) \quad \text{or} \quad \sq \in W(\{\aq_{i+1}, \aq_{i+2}, \dots, \aq_n\}).
    \end{equation*}
    This implies that $\aq_1$ occurs in $\sq$ and $i \in \{2, 3, \dots, k\}$.
    Hence $\sq \in W(\{\aq_1, \aq_2, \dots, \aq_{i-1}\}) \subseteq W(\{\aq_1, \aq_2, \dots, \aq_{k-1}\})$.
    It follows that $\aq_k$ does not occur in $\sq$.
    Hence, the right $\aq_i$ of the subword $\aq_i \sq \aq_i$ does not occur strictly to the right of $\aq_k$ in $\eq_{\{2, 3, \dots, k\}}$, and therefore
    we must have $i = k$.
    This implies that we can use a right deletion to delete the occurrence of the letter $\aq_k$ in the suffix $\eq_{\{2, 3, \dots, k\}}$ of the word $\zeta(x) \aq_1 \eq_{\{2, 3, \dots, k\}}$.
    Thus
    \begin{equation*}
        \zeta(x) \aq_1 \eq_{\{2, 3, \dots, k\}} \to \zeta(x) \aq_1 \eq_{\{2, 3, \dots, k-1\}}.
    \end{equation*}
    This implies
    \begin{equation*}
        x a_1 e_{\{2, 3, \dots, k\}} = x a_1 e_{\{2, 3, \dots, k-1\}}.
    \end{equation*}
    Using Lemma~\ref{lemma:shorten}~(\ref{lemma:shorten:ii}) and the definition of $k = m(x)$, we obtain
    \begin{align*}
        x e_{\{1, 2, \dots, k-1\}} &= x e_{\{2, 3, \dots, k-1\}} a_1 \\
        &= x a_1 e_{\{2, 3, \dots, k-1\}} a_1 \\
        &= x a_1 e_{\{2, 3, \dots, k\}} a_1 \\
        &= x e_{\{2, 3, \dots, k\}} a_1 \\
        &= x e_{\{1, 2, \dots, k\}} \\
        &= f.
    \end{align*}
    This contradicts the definition of $k = m(x)$.
    Hence, the word $\zeta(x) \aq_1 \eq_{\{2, 3, \dots, k\}}$ must be canonical.

    Since $x a_1 e_{\{2, 3, \dots, k\}} = \varphi(\zeta(x) \aq_1 \eq_{\{2, 3, \dots, k\}})$ and $\zeta(x) \aq_1 \eq_{\{2, 3, \dots, k\}}$ is canonical, it follows that
    \[
    \zeta(x a_1 e_{\{2, 3, \dots, k\}}) = \zeta(x) \aq_1 \eq_{\{2, 3, \dots, k\}},
    \]
    which concludes the proof.
\end{proof}

Now we obtain an explicit description of $T$.
\begin{theorem}\label{theorem:T}
    We have
    \begin{equation*}
        T = \{ x a_1 e_{\{2, 3, \dots, m(x)\}} \mid x \in \langle a_2, a_3, \dots, a_n \rangle\}.
    \end{equation*}
\end{theorem}

\begin{proof}
    Let $x \in \langle a_2, a_3, \dots, a_n \rangle$.
    By Theorem~\ref{Theorem_M}, we have
    \[
    \zeta(x a_1 e_{\{2, 3, \dots, m(x)\}}) = \zeta(x) \aq_1 \eq_{\{2, 3, \dots, m(x)\}},
    \]
    which shows $x a_1 e_{\{2, 3, \dots, m(x)\}} \in K_n^1$.
    Since $x \in \langle a_2, a_3, \dots, a_n \rangle$, it follows that $x \neq f$, which implies $m(x) \geq 1$.
    By Lemma~\ref{lemma:shorten}~(\ref{lemma:shorten:ii}), we further have  
    \[
    x a_1 e_{\{2, 3, \dots, m(x)\}} a_1 = x e_{\{2, 3, \dots, m(x)\}} a_1 = x e_{\{1, 2, \dots, m(x)\}} = f,
    \]
    by the definition of $m(x)$.
    This proves 
    \begin{equation}\label{z}
        \{ x a_1 e_{\{2, 3, \dots, m(x)\}} \mid x \in \langle a_2, a_3, \dots, a_n \rangle\} \subseteq T.
    \end{equation}

    Let $x, y \in \langle a_2, a_3, \dots, a_n \rangle$ and assume that $x \neq y$.
    Lemma~\ref{lemma:canonical_unique} implies $\zeta(x) \neq \zeta(y)$.
    Since $\zeta(x), \zeta(y) \in W(\{\aq_2, \aq_3, \dots, \aq_n\})$, it follows that 
    \[
    \zeta(x) \aq_1 \eq_{\{2, 3, \dots, m(x)\}} \neq \zeta(y) \aq_1 \eq_{\{2, 3, \dots, m(y)\}}.
    \]
    Theorem~\ref{Theorem_M} further implies 
    \[
    \zeta(x a_1 e_{\{2, 3, \dots, m(x)\}}) \neq \zeta(y a_1 e_{\{2, 3, \dots, m(y)\}}),
    \]
    and hence we have 
    \[
    x a_1 e_{\{2, 3, \dots, m(x)\}} \neq y a_1 e_{\{2, 3, \dots, m(y)\}}.
    \]
    This implies that
    \[
    |\{ x a_1 e_{\{2, 3, \dots, m(x)\}} \mid x \in \langle a_2, a_3, \dots, a_n \rangle\}| = |\langle a_2, a_3, \dots, a_n \rangle| = |K_{n-1}|.
    \]
    By Corollary~\ref{bc1}, we have $|T| = |K_{n-1}|$ as well.
    Then~\eqref{z} implies 
    \[
    T = \{ x a_1 e_{\{2, 3, \dots, m(x)\}} \mid x \in \langle a_2, a_3, \dots, a_n \rangle\},
    \]
    which completes the proof.
\end{proof}

\begin{theorem}\label{R_rule}
    The set $R$ is a subsemigroup of $K_n$.
    Moreover, for $x, y \in R$, we have 
    \[
    x y = 
    \begin{cases}
        e_{\{2, 3, \dots, n\}}, & x = y = e_{\{2, 3, \dots, n\}}, \\
        x, & x \in T, y = e_{\{2, 3, \dots, n\}}, \\
        f, & y \in T.
    \end{cases}
    \]
\end{theorem}
\begin{proof}
    Let $x, y \in R$. Then 
    \[
    (x y) a_1 = x y a_1 = x f = f,
    \]
    which shows $x y \in R$, and hence $R$ is a subsemigroup of $K_n$.

    If $x = y = e_{\{2, 3, \dots, n\}}$, then 
    \[
    x y = e_{\{2, 3, \dots, n\}} e_{\{2, 3, \dots, n\}} = e_{\{2, 3, \dots, n\}},
    \]
    since $e_{\{2, 3, \dots, n\}}$ is an idempotent in $K_n$.

    If $x \in T$ and $y = e_{\{2, 3, \dots, n\}}$, then by Theorem~\ref{theorem:T}, we have
    \begin{equation}\label{x1}
        x = z a_1 e_{\{2, 3, \dots, m(z)\}},
    \end{equation}
    for some $z \in \langle a_2, a_3, \dots, a_n \rangle$.
    Therefore,
    \[
    xy = z a_1 e_{\{2, 3, \dots, m(z)\}} e_{\{2, 3, \dots, n\}} = z a_1 e_{\{2, 3, \dots, n\}}.
    \]
    Since 
    \[
    z a_1 e_{\{2, 3, \dots, n\}} = \varphi(\zeta(z) \aq_1 \eq_{\{2, 3, \dots, n\}}),
    \]
    $\zeta(z), \eq_{\{2, 3, \dots, n\}} \in W(\{\aq_2, \aq_3, \dots, \aq_n\})$, and $\zeta(z)$ is canonical,
    Lemma~\ref{lc0} implies 
    \begin{equation}\label{z1}
        \zeta(xy) = \zeta(z a_1 e_{\{2, 3, \dots, n\}}) = \can(\zeta(z) \aq_1 \eq_{\{2, 3, \dots, n\}}) = \zeta(z) \aq_1 \uq^*,
    \end{equation}
    for some $\uq^* \leq \eq_{\{2, 3, \dots, n\}}$.
    By~\eqref{z1}, we have $1 \in c(x y)$.
    Therefore $x y \in K_n^1$ and since $x y \in R$, we must have $x y \in T$.
    By Theorem~\ref{theorem:T}, we further have
    \begin{equation}\label{x2}
        x y = z' a_1 e_{\{2, 3, \dots, m(z')\}},
    \end{equation}
    for some $z' \in \langle a_2, a_3, \dots, a_n \rangle$.
    By Theorem~\ref{Theorem_M}, we have 
    \begin{equation}\label{z2}
        \zeta(xy) = \zeta(z' a_1 e_{\{2, 3, \dots, m(z')\}}) = \zeta(z') \aq_1 \eq_{\{2, 3, \dots, m(z')\}}.
    \end{equation}
    Comparing~\eqref{z1} and~\eqref{z2}, and noting that $\zeta(z), \zeta(z') \in W(\{\aq_2, \aq_3, \dots, \aq_n\})$, it follows that 
    \[
    \zeta(z) = \zeta(z'),
    \]
    which implies $z = z'$ by Lemma~\ref{lemma:canonical_unique}.
    By~\eqref{x1} and~\eqref{x2}, it then follows that
    \[
    xy = x.
    \]

    If $y \in T$, then by Theorem~\ref{theorem:T}, we have 
    \[
    y = z a_1 e_{\{2, 3, \dots, m(z)\}},
    \]
    for some $z \in \langle a_2, a_3, \dots, a_n \rangle$.
    Then by Lemma~\ref{lemma:shorten}~(\ref{lemma:shorten:ii}), we have
    \[
    xy = x z a_1 e_{\{2, 3, \dots, m(z)\}} = x a_1 z a_1 e_{\{2, 3, \dots, m(z)\}} = f z a_1 e_{\{2, 3, \dots, m(z)\}} = f.
    \]
    This concludes the proof.
\end{proof}

Observe that the cases for the multiplication rule in $R$ in Theorem~\ref{R_rule} are exhaustive by~\eqref{R_decomp}.

\section{Parity of $|K_n|$}\label{PAR}

In Section~\ref{TE}, we used a simple decomposition of $K_n$
\begin{equation}\label{kn_decomp}
    K_n = \langle a_2, a_3, \dots, a_n \rangle \cup K_n^1,
\end{equation}
where $K_n^1$ is the set of all $x \in K_n$ whose canonical form contains $\aq_1$.
This idea leads us to the following result. 
\begin{theorem}\label{Parity:Theorem}
    For any $k \geq 1$, $|K_{2k+1}|$ is even and $|K_{2k}|$ is odd. 
\end{theorem}

Let $S \subseteq K_n$ and let $X \subseteq \{1, 2, \dots, n\}$.
Then we define
\begin{equation*}
    S^X = \{x \in S \mid X \subseteq c(x)\}.
\end{equation*}
Observe that $K_n^1 = K_n^{\{1\}}$.

We also define the map $t : W(A) \to W(A)$ by
\begin{equation*}
    t(\aq_{i_1} \aq_{i_2} \dots \aq_{i_k}) = \aq_{n-i_1 +1} \aq_{n-i_2+1} \dots \aq_{n-i_k+1},
\end{equation*}
for all $k \geq 1$ and $i_1, i_2, \dots, i_k \in \{1, 2, \dots, n\}$, and $t(\eq) = \eq$.
Then $t(t(\wq)) = \wq$ for every $\wq \in W(A)$, and hence $t$ is a bijection with $t^{-1} = t$.

\begin{lemma}\label{t:lemma}
    Let $\wq \in W(A)$. Then $\wq$ is canonical if and only if $t(\wq)$ is canonical.
\end{lemma}

\begin{proof}
    Assume that $\wq$ is canonical.
    Suppose there exists a subword of the form $\aq_i \sq \aq_i$ in $t(\wq)$.
    Then
    \[
    t^{-1}(\aq_i \sq \aq_i) = t(\aq_i \sq \aq_i) = \aq_{n-i+1} t(\sq) \aq_{n-i+1}
    \]
    is a subword of $\wq$.
    Since $\wq$ is canonical, there exist $k$ and $j$ with $k > n-i+1$ and $j < n-i+1$ such that $\aq_k$ and $\aq_j$ occur in $t(\sq)$.
    It follows that $t^{-1}(\aq_k) = t(\aq_k) = \aq_{n-k+1}$ and $t^{-1}(\aq_j) = t(\aq_j) = \aq_{n-j+1}$ occur in $\sq$.
    Since $k > n-i+1$, we have $n-k+1<i$, and since $j<n-i+1$, we have $n-j+1>i$.
    Hence $t(\wq)$ is canonical.

    Conversely, if $t(\wq)$ is canonical, then applying the above argument to $t(\wq)$ yields that $t(t(\wq)) = \wq$ is canonical.
\end{proof}

\begin{proof}[Proof of Theorem~\ref{Parity:Theorem}]
    Assume that $n \geq 3$.
    We decompose $K_n$ as
    \begin{equation*}
        K_n = \langle a_2, a_3, \dots, a_{n-1} \rangle \cup \langle a_1, a_2, \dots, a_{n-1} \rangle^{\{1\}} 
        \cup \langle a_2, a_3, \dots, a_n \rangle^{\{n\}} \cup \langle a_1, a_2, \dots, a_n \rangle^{\{1, n\}},
    \end{equation*}
    where the union is disjoint.
    To justify this, observe that for every $x \in K_n$, its canonical form 
    $\zeta(x)$ either lies in $W(\{\aq_2, \aq_3, \dots, \aq_{n-1}\})$, 
    in which case $x \in \langle a_2, a_3, \dots, a_{n-1} \rangle$, or it contains at least one of the letters $\aq_1$ or $\aq_n$.

    If it contains $\aq_1$ but not $\aq_n$, then 
    $x \in \langle a_1, a_2, \dots, a_{n-1} \rangle^{\{1\}}$.
    If it contains $\aq_n$ but not $\aq_1$, then 
    $x \in \langle a_2, a_3, \dots, a_n \rangle^{\{n\}}$.
    Finally, if it contains both $\aq_1$ and $\aq_n$, then 
    $x \in \langle a_1, a_2, \dots, a_n \rangle^{\{1, n\}}$.

    The semigroup $\langle a_2, a_3, \dots, a_{n-1} \rangle$ is in bijection with $K_{n-2}$.
    Moreover, $\langle a_1, \dots, a_{n-1} \rangle^{\{1\}}$ and $\langle a_2, \dots, a_n \rangle^{\{n\}}$ are also in bijection.
    A corresponding bijection is $\varphi \circ t \circ \zeta$.
    Since $\langle a_2, \dots, a_n \rangle$ is in bijection with $K_{n-1}$, using~\eqref{kn_decomp} we obtain
    \begin{equation*}
        |K_n^1| = |K_n| - |\langle a_2, a_3, \dots, a_n \rangle| =  |K_n| - |K_{n-1}|,
    \end{equation*}
    and therefore 
    \begin{equation*}
       |\langle a_1, a_2, \dots, a_{n-1} \rangle^{\{1\}}| = |K_{n-1}^1| = |K_{n-1}| - |K_{n-2}|.
    \end{equation*}
    Hence we have
    \begin{align*}
        |K_n| &= |\langle a_2, a_3, \dots, a_{n-1} \rangle| + |\langle a_1, a_2, \dots, a_{n-1} \rangle^{\{1\}}| \\
        &\quad + |\langle a_2, a_3, \dots, a_n \rangle^{\{n\}}| + |\langle a_1, a_2, \dots, a_n \rangle^{\{1, n\}}| \\
        &= |K_{n-2}| + 2 |\langle a_1, a_2, \dots, a_{n-1} \rangle^{\{1\}}| + |\langle a_1, a_2, \dots, a_n \rangle^{\{1, n\}}| \\
        &= |K_{n-2}| + 2 (|K_{n-1}| - |K_{n-2}|) + |\langle a_1, a_2, \dots, a_n \rangle^{\{1, n\}}|.
    \end{align*}

    We identify $\langle a_1, a_2, \dots, a_n \rangle^{\{1, n\}}$ with the set of canonical forms of the elements of $\langle a_1, a_2, \dots, a_n \rangle^{\{1, n\}}$, that is, with the set of all canonical words over $A$ containing the letters $\aq_1$ and $\aq_n$. 
    We denote this set by $V$.
    By Proposition~\ref{prop:word_bound}, $\aq_1$ and $\aq_n$ occur exactly once in each element of $V$.
    Thus, we can write
    \begin{align}
        V = &\{\uq \aq_1 \vq \aq_n \wq \mid \uq, \vq, \wq \in W(\{\aq_2, \aq_3, \dots, \aq_{n-1}\}), \uq \aq_1 \vq \aq_n \wq \text{ is canonical}\} \cup 
        \label{v1} \\
        &\{\uq \aq_n \vq \aq_1 \wq \mid \uq, \vq, \wq \in W(\{\aq_2, \aq_3, \dots, \aq_{n-1}\}), \uq \aq_n \vq \aq_1 \wq \text{ is canonical}\},
        \label{v2}
    \end{align}
    where the union is disjoint.
    Denote by $V_1$ and $V_2$ the sets in~\eqref{v1} and~\eqref{v2}, respectively.

    We show that the restriction $t|_{V_1} : V_1 \to V_2$ is a bijection.
    If $\uq \aq_1 \vq \aq_n \wq \in V_1$, then $t(\uq \aq_1 \vq \aq_n \wq)$ is canonical by Lemma~\ref{t:lemma}.
    Moreover, we have 
    \begin{equation*}
        t(\uq \aq_1 \vq \aq_n \wq) = t(\uq) \aq_n t(\vq) \aq_1 t(\wq)
    \end{equation*}
    and $t(\uq), t(\vq), t(\wq) \in W(\{\aq_2, \dots, \aq_{n-1}\})$.
    Hence $t|_{V_1}$ maps into $V_2$.
    Since $t$ is injective, so is $t|_{V_1}$. 
    Let $\uq \aq_n \vq \aq_1 \wq \in V_2$.
    Then
    \begin{equation*}
        t(\uq \aq_n \vq \aq_1 \wq) = t(\uq) \aq_1 t(\vq) \aq_n t(\wq)
    \end{equation*}
    is canonical by Lemma~\ref{t:lemma}, and since $t(\uq), t(\vq), t(\wq) \in W(\{\aq_2, \dots, \aq_{n-1}\})$, we have
    $t(\uq \aq_n \vq \aq_1 \wq) \in V_1$.
    Since $\uq \aq_n \vq \aq_1 \wq = t(t(\uq \aq_n \vq \aq_1 \wq))$, surjectivity follows.
    Hence $|V_1| = |V_2|$ and 
    \begin{equation*}
       |\langle a_1, a_2, \dots, a_n \rangle^{\{1, n\}}| = |V| = 2|V_1|. 
    \end{equation*}
    
    Thus,
    \begin{equation*}
        |K_n| = |K_{n-2}| + 2 (|K_{n-1}| - |K_{n-2}|) + 2|V_1|.
    \end{equation*}
    It follows that 
    \begin{equation*}
        2 \mid |K_n| - |K_{n-2}|.
    \end{equation*}
    Therefore,
    \[
    |K_n| \equiv |K_{n-2}| \pmod{2}.
    \]

    Since $K_1 = \{e, a_1\}$, we have $|K_1| = 2$.
    Therefore, $|K_{2k+1}|$ is even.
    We have $K_2 = \{e, a_1, a_2, a_1 a_2, a_2 a_1\}$, hence $|K_2| = 5$.
    It follows that $|K_{2k}|$ is odd. 
\end{proof}

\bmhead{Acknowledgements}

I would like to thank my close family and friends.
This work would not have been possible without your unconditional support, relentless motivation, and loving relationships.
Thank you.

\bibliography{sn-bibliography.bib}

\end{document}